\documentclass{article}

\usepackage{graphicx}
\usepackage{amsmath}
\usepackage{natbib}
\usepackage{amsfonts}
\usepackage{amssymb}  
\usepackage{amsthm}  
\usepackage{subfigure}
\usepackage{verbatim}
\usepackage{color}
\usepackage{soul}
																							
\newcommand{\de}{\text{d}}
\newcommand{\R}{\mathbb{R}}

\newcommand{\N}{\mathbb{N}}

\newtheorem{theorem}{Theorem}

\newtheorem{lemma}{Lemma}
\newtheorem{remark}{Remark}
\newtheorem{example}{Example}

\title{Bayesian multivariate mixed-scale density estimation}
\author{Antonio Canale\thanks{Department of Economics and Statistics, University of Turin, and Collegio Carlo Alberto} and David B. Dunson\thanks{Department of Statistical Science, Duke University}}

\begin{document}
\maketitle

\begin{abstract}
Although continuous density estimation has received abundant attention in the Bayesian nonparametrics literature, there is limited theory on multivariate mixed scale density estimation.  In this note, we consider a general framework to jointly model continuous, count and categorical variables under a nonparametric prior, which is induced through rounding latent variables having an unknown density with respect to Lebesgue measure.  For the proposed class of priors, we provide sufficient conditions for large support, strong consistency and rates of posterior contraction. These conditions allow one to convert sufficient conditions obtained in the setting of multivariate continuous density estimation to the mixed scale case. To illustrate the procedure a rounded multivariate nonparametric mixture of Gaussians is introduced and applied to a crime and communities dataset.
\end{abstract}

{\center \textbf{Keywords: }}
Large support; Mixed discrete and continuous; Nonparametric Bayes; Rate of posterior contraction; Strong posterior consistency

\section{Introduction}

In this paper we focus on nonparametric models for estimating unknown joint distributions for mixed scale data consisting of binary, ordered categorical, continuous and count measurements.  Somewhat surprisingly given the considerable applied interest, the literature on nonparametric estimation for mixed scale data is very small.  From a frequentist kernel smoothing perspective,  Li, Racine and co-authors   \citep{li:racine:2003, hall:etal:2004, ouya:etal:2006, li:racine:2008} proposed mixed kernel methodology and considered properties under somewhat restrictive conditions. These conditions are  relaxed  by  \citet{efro:2011} and a data-driven estimator designed to combat the curse of dimensionality is proposed.  His work assumed compact support for continuous variables and bounded support for discrete variables.  A recent collection of frequentist contributions to the topic can be found in \citet{adel:2013}. From a Bayesian semiparametric perspective, \citet{nore:pele:2012}  show posterior consistency for a finite mixture of latent multivariate normals, assuming bounded support for the discrete variables. 
Similar models have been applied for mixed scale data, but without theory support \citep{ever:1988,morl:2012,song:etal2009}. 

In the parametric literature on mixed scale modeling, it is common to model the joint distribution of underlying variables as Gaussian, with the categorical variables then obtained via thresholding.  A number of authors have considered variations on this theme in the nonparametric case, via modeling one or more components as non-Gaussian using mixtures and other approaches.  We apply a related strategy here to obtain a broad framework, with our focus then on studying the theory related to 
 large support, posterior consistency and rates of convergence.  This is the first contribution (to our knowledge) to Bayesian posterior consistency and posterior rates of contraction for a large class of mixed scale models. In particular a minimax rate for this class of 
problems is not known. However it is potentially faster than usual rates for estimating smooth continuous densities.  
Our goal is to provide theorems that allow leveraging on results obtained for multivariate continuous densities. We consider a multivariate mixed scale generalization of the rounding framework of \citet{cana:duns:2011}. This extension is intuitive both from a practical and theoretical point of view.

Section 2 introduces preliminaries, Section 3 proposes the class of priors under consideration, and Section 4 presents theorems on the KL support of the prior, strong posterior consistency and rates of posterior contraction.  Section 5 discusses an application to US communities and crime dataset, using a particular prior specification.

\section{Preliminaries and notation}
\label{sec:prelim}
Our focus is on modeling of joint probability distributions of mixed scale data \mbox{$y = (y_1^T, y_2^T)^T$}, where $y_1 =(y_{1,1},\dots, y_{1,p_1}) \in \mathcal{Y} \subseteq \R^{p_1}$ is a $p_1 \times 1$ vector of continuous observations and
$y_2  = (y_{2,p_1+1},\dots, y_{2,p}) \in Q$ with $Q = \bigotimes_{j=1}^{p_2}\{0, 1, \dots, q_{j} - 1 \}$ is a $p_2 \times 1$ vector of discrete variables having $q = (q_1,\ldots,q_{p_2})^T$ as the respective number of levels and $p_2 = p - p_1$. Clearly $y_2$ can include binary variables ($q_j = 2$), categorical variables ($q_j >2$) or counts ($q_j = \infty$). Hence, $y$ is a $p \times 1$ vector of variables having mixed measurement scales.  
We let $y \sim f$, with $f$ denoting the joint density with respect to an appropriate dominating measure $\mu$ to be defined below.  The set of all possible such joint densities is denoted $\mathcal{F}$.  Following a Bayesian nonparametric approach, we propose to specify a prior $f \sim \Pi$ for the joint density having \mbox{large support over $\mathcal{F}$.}

For the continuous variables, we let $(\Omega_1,\mathcal{S}_1,\mu_1)$ denote the $\sigma$-finite measure space having $\Omega_1 = \mathcal{Y}$, $\mathcal{S}_1$ the Borel $\sigma$-algebra of subsets of $\Omega_1$, and $\mu_1$ the Lebesgue measure. Similarly for the discrete variables we let $(\Omega_2,\mathcal{S}_2,\mu_2)$ denote the $\sigma$-finite measure space having $\Omega_2 \subseteq \N^{p_2}$, a subset of the $p_2$-dimensional set of natural numbers, $\mathcal{S}_2$ containing all non-empty subsets of $\Omega_2$, and $\mu_2$ the counting measure.  
Then, we let $\mu = \mu_2 \times \mu_2$ be the product measure on the product space $(\Omega,\mathcal{S}) = (\Omega_1,\mathcal{S}_1) \times (\Omega_2,\mathcal{S}_2)$.  To formally define the joint density $f$, first let $\nu$ denote a $\sigma$-finite measure on $(\Omega, \mathcal{S})$ that is absolutely continuous with respect to $\mu$.  Then, by the Radon-Nikodym theorem, there exists a function $f$ such that $\nu(A) = \int_A f d\mu$.

In studying properties of a prior $\Pi$ for the unknown density $f$, such as large support and posterior consistency, it is necessary to define notions of distance and neighborhoods within the space of densities $\mathcal{F}$.  Letting $f_0 \in \mathcal{F}$ denote an arbitrary density, such as the true density that generated the data, the Kullback-Leibler divergence of $f$ from $f_0$ is
\begin{align*}
	d_{KL}(f_0,f) & = \int_\Omega f_0 \log(f_0/f) d \mu = \int_{\Omega_1} \int_{\Omega_2} f_0 \log(f_0/f) d \mu_1\, d \mu_2\, \notag \\
			& = \int_{\mathcal{Y}} \sum_{y_2 \in Q} f_0(y_1,y_2 ) \log\left(\frac{f_0(y_1,y_2 )}{f(y_1,y_2 )}\right) d y_1 
\end{align*}
with the integrals taken in any order from Fubini's theorem. Another topology is induced by the $L_1$-metric. 
If $f$ and $f_0$ are probability distributions with respect to the product measure $\mu$, their $L_1$-distance is
\begin{align*}
	||f_0 - f|| & = \int_\Omega |f_0 - f|  d \mu =  \int_{\Omega_1} \int_{\Omega_2}  |f_0 - f| d \mu_1\, d \mu_2\    \notag \\
		& =  \int_{\mathcal{Y}} \sum_{y_2 \in Q}  |f_0(	y_1,y_2 ) - f(y_1,y_2)| d y_1.
\end{align*}

\section{Rounding prior}
\label{sec:roundprior}

In order to induce a prior $f \sim \Pi$ for the density of the mixed scale variables, we let
\begin{eqnarray}
y = h(y^*),\quad y^* \sim f^*,\quad f^* \sim \Pi^*, \label{eq:round}
\end{eqnarray}
where $h: \R^p \to \Omega$,  $y^* = (y_1^*,\ldots, y_p^*)^T \in \R^p$, $f^* \in \mathcal{F}^*$, $\mathcal{F}^*$ is the set of densities with respect to Lebesgue measure over $\R^p$, and
$\Pi^*$ is a prior over $\mathcal{F}^*$. To introduce an appropriate mapping $h$, we let 
\begin{eqnarray}
h(y^*) = \big\{ h_1(y_1^*)^T, h_2(y_2^*)^T \big\}^T, \label{eq:h}
\end{eqnarray}
where $h_1(y_1^*)=\{ h_{1,1}(y_{1,1}^*), \dots, h_{1,p_1}(y_{1,p_1}^*)\}$, $h_{1,j}: \R \to \mathcal{Y}_j \subseteq \R$ is a monotone one-to-one differentiable mapping, with $\mathcal{Y}_j$ the support of $y_{1,j}$, and $h_2$ are thresholding functions that replace the real-valued inputs with non-negative integer outputs by thresholding the different inputs separately. 
Let $A^{(j)} = \{ A_1^{(j)},\ldots, A_{q_j}^{(j)} \}$ denote a prespecified partition of $\R$ into $q_j$ mutually exclusive subsets, for $j=1,\ldots, p_2$, with the subsets ordered so that $A_h^{(j)}$ is placed before $A_l^{(j)}$ for all $h < l$.  Then, letting $A_{y_2} = \{ y_2^*: y_{2,j}^* \in A_{y_{2j}}^{(j)}, j=1,\ldots, p_2 \}$, the mixed scale density $f$ is defined as
\begin{eqnarray}
f(y) = g(f^*) 
		  &=& \int_{A_{y_2}} f^*(h_1^{-1}(y_1), y^*_2) \, |J_{h_1^{-1}(y_1)}| \, dy_2^*\label{eq:g}
\end{eqnarray}
where $J_{h_1^{-1}(y_1)}$ is the Jacobian matrix of the inverse function $h_1^{-1}$. A typical choice for $h_{1,j}$ when $\mathcal{Y}_j = \R$ is the identity link which has the benefit to greatly simplify the formulation.
%
The function $g: \mathcal{F^*} \to \mathcal{F}$ defined in (\ref{eq:g}) is a bijective mapping from the space of densities with respect to Lebesgue measure on $\R^p$ to the space of mixed-scale densities $\mathcal{F}$. It is clear that there are infinitely many $f^*$ that map into a single $g(f^*) = f_0$.
This framework generalizes \citet{cana:duns:2011}, which focused only on count variables.  The theory is substantially more challenging in the mixed scale case when there are continuous variables involved.  

\section{Theoretical properties}
\label{theoretical}

Clearly the properties of the induced prior $f \sim \Pi$ will be driven largely by the properties of $f^* \sim \Pi^*$. 
Lemma 1 shows that the mapping $g: \mathcal{F}^* \to \mathcal{F}$ maintains Kullback-Leibler (KL) neighborhoods.
The proof is omitted as being a straightforward modification of that for Lemma 1 in  \citet{cana:duns:2011}.

\begin{lemma}
Choose any $f_0^*$ such that $f_0 = g(f^*_0)$ for any fixed $f_0 \in \mathcal F$. Let $\mathcal{K}_\epsilon (f^*_0)= \{ f^* : d_{KL}(f^*_0,f^*) < \epsilon\}$ be a Kullback-Leibler neighborhood of size $\epsilon$ around $f^*_0$.  Then the image $g(\mathcal{K}_{\epsilon}(f_0^*))$ contains values $f \in \mathcal{F}$ in a Kullback-Leibler neighborhood of $f_0$ of at most size $\epsilon$.
\label{lem:klmapping}
\end{lemma}

Large support of the prior plays a crucial role in posterior consistency.
Under the theory of Schwartz \citep{schw:1965}, given $f_0$ in the KL support of the prior, strong posterior consistency can be obtained by showing the existence of an exponentially consistent sequence of tests for the hypothesis $H_0:f = f_0$ versus $H_1: f \in U^C(f_0)$ where $U(f_0)$ is a neighborhood of $f_0$ and $U^C(f_0)$ is the complement of $U(f_0)$. \citep{ghos:etal:1999} show that the existence of such a sequence of tests is guaranteed by balancing the size of a sieve and the prior probability assigned to its complement. 

We now provide sufficient conditions for $L_1$ posterior consistency for priors in the class proposed in expression (1).  Our Theorem \ref{theo:consistency} builds on Theorem 8 of \citet{ghos:etal:1999}. 
The main differences are that we define the sieve $\mathcal{F}_n$ as $g(\mathcal{F}^*_n)$, where  $\mathcal{F}^*_n$ is a sieve on $\mathcal{F}^*$ and that we require conditions on the prior probability in terms of the underlying $\Pi^*$. The proof relies on the same steps of  \citet{ghos:etal:1999} given lemmas \ref{lem:l1mapping} and \ref{lem:metricentropy} (reported in the Appendix) which give an upper bound for the $L_1$ metric entropy $J(\delta,\mathcal{F}_n)$ defined as the logarithm of the minimum number of $\delta$-sized $L_1$ balls needed to cover $\mathcal{F}_n$. 
%
\begin{theorem}
Let $\Pi$ be a prior on $\mathcal F$ induced by $\Pi^*$ as described in expression (\ref{eq:round}). Suppose $f_0$ is in the KL support of $\Pi$ and let $U =\{f \in \mathcal{F}: ||f-f_0||< \epsilon \}$. If for each $\epsilon>0$, there is a $\delta < \epsilon$, $c_1$, $c_2 >0$, $\beta < \epsilon^2/8$ and there exist sets $\mathcal{F}_n^* \subset \mathcal{F}^*$ 
such that for $n$ large
\begin{itemize}
	\item[$(i)$] $\Pi^*(\mathcal{F}_n^{*C}) \leq c_1 e^{-n c_2}$;
	\item[$(ii)$] $J(\delta,  \mathcal{F}_n^*) < n\beta$
\end{itemize}
then $\Pi(U \mid {\bf y}_1, \dots, {\bf y}_n) \to 1$ a.s. $P_{f_0}$.
\label{theo:consistency}
\end{theorem}

We now state a theorem on the rate of convergence (contraction) of the posterior distribution.  The theorem gives conditions on the prior $\Pi^*$ similar to those directly required by Theorem 2.1 of  \citet{ghos:etal:2000}. The proof is reported in the Appendix.

\begin{theorem}
Let $\Pi$ be the prior on $\mathcal F$ induced by $\Pi^*$ as described in expression (\ref{eq:round}) and $U = \{f : d(f,f_0) \leq M \epsilon_n\}$ with $d$ the $L_1$ or Hellinger distance.
Suppose that for a sequence $\epsilon_n$, with $\epsilon_n \to 0$ and $n\epsilon_n^2 \to \infty$, a constant $C>0$, sets $\mathcal{F}_n^* \subset \mathcal{F}^*$ 
and $B^*_n = \{f^* : \int f_0^{*} \log(f_0^{*}/f^*) d\mu \leq \epsilon^2_n, \int f_0^{*} (\log(f_0^{*}/f^*))^2 d\mu \leq \epsilon^2_n\}$ defined for a given $f_0^{*} \in g^{-1}(f_0)$, we have
\begin{itemize}
	\item[$(iii)$] $J(\epsilon_n,  \mathcal{F}_n^*) < C n \epsilon^2_n$;
	\item[$(iv)$] $\Pi^*(\mathcal{F}_n^{*C}) \leq \exp\{- n \epsilon^2_n (C + 4) \}$;
	\item[$(v)$] $\Pi^*(B^*_n) \geq  \exp\{-C n \epsilon^2_n \}$
\end{itemize}
then for sufficiently large $M$, we have that $\Pi(U^C \mid {\bf y}_1, \dots, {\bf y}_n) \to 0$ in $P_{f_0}$-probability.
\label{theo:rate}
\end{theorem}

\begin{remark}
Since $g$ is bijective there are infinitely many $f_0^* \in g^{-1}(f_0)$ but it is sufficient that condition $(v)$ is satisfied for just one of them.
\end{remark}

The convergence rate that can obtained using Theorem~\ref{theo:rate} may change with respect to the particular choice for $h_1$. Assume the latent variables $y_1^*$ are drawn from an $\alpha$-H\"older smooth density.  Since the smoothness of the density of the observed continuous variables $y_1$ depends on the mapping $h_1$, the choice for $h_1$ can decrease or increase the smoothness, impacting the optimal rate. Such complications clearly do not arise if $h_1$ is the identity function. 

The rate obtained using Theorem~\ref{theo:rate} in general does not correspond to the minimax optimal rate in this class of problems, but represents an upper bound on the rate.  If the $p_2$ categorical variables all have finite support, the minimax rate is shown in Lemma~\ref{lem:minimax} below.

\begin{lemma}
Let $q_j < \infty$ for all $j = 1, \dots, p_2$ and assume that the $p_1$ continuous variables have marginal $\alpha$-H\"older smooth density. 
Then, the minimax optimal rate for the mixed-scale density is $n^{-\alpha/(2\alpha + p_1)}$.
\label{lem:minimax}
\end{lemma}

\begin{example}
Conditions $(iii)$--$(v)$ are satisfied, for example, by a Dirichlet process mixture of multivariate Gaussians prior as discussed in \citet{shen:etal} for any $f_0^*$ belonging to the smoothness class of locally $\alpha$-H\"older functions. This convergence rate result for multivariate continuous density estimation directly implies the convergence rate for the mixed scale density with conditions on the first $p_1$ components. In particular if $h_1$ is the identity function, the requirements for $f_0^*$ to be in the KL support of $\Pi^*$ induce the same requirements for the first $p_1$ components of $f_0$ with no condition on the remaining $p_2$ discrete components.
\end{example}

\section{Application to crime data}

We use our proposed methodology to estimate the joint density of per capita income, in thousands of \$ ($y_1$) and number of murders in 1990 ($y_2$) in the US. The dataset is part of a bigger dataset on communities and crime from the UCI Machine Learning Repository. The data set is from the 1990 US Census, 1995 US FBI Uniform Crime Report and 1990 US Law Enforcement Management and Administrative Statistics Survey. Our aim is to estimate the joint mixed-scale density of the per capita income (continuous) and number of murders (counts) in each state with more than 20 observations to illustrate our method and study the relationship between these two variables.
For each state the pair $(y_{1,i}, y_{2,i})^T$ is available where $i=1, \dots n_j$ and $n_j$ is the number of communities present in the dataset for state $j$.  
This analysis is clearly illustrative since the FBI noted that even the use of the complete dataset is over-simplistic if one wants to evaluate communities, since many relevant factors are not included. 

To model these data we define our mixed-scale prior through a latent Dirichlet process (DP) location-scale mixture of Gaussians prior \citet{esco:west:1995,mull:etal:1996}. Let $\Pi^*$ be the prior induced by the model
\begin{equation}
f^*(y^*) = \int N(y^*; \theta, \Sigma) d G(\theta, \Sigma), \, \, G \sim DP(\alpha P_0),
\label{eq:model}
\end{equation}
where $P_0 = \text{N}_p(\theta;\theta_0,\kappa_0 \Sigma)\text{Inv-W}(\Sigma;\nu_0,{\bf S}_0)$ is a normal-inverse-Wishart base measure and $\alpha >0$ is the DP scale parameter. This multivariate location-scale mixture is a default choice for multivariate density estimation in many contexts \citep{mull:etal:1996,MacE:Mull:1998} and has been recently shown to lead to posterior consistency \citep{cana:debl:2013}. The latent prior specification is completed eliciting the prior hyperparameters, which we fix equal to $\alpha=1$, $\nu_1 = \nu_2=3$, $\kappa_0=1$, ${\bf S}_0=\mbox{diag}(6,60)$, and $\theta_0= \bar{y}$, where, following an empirical Bayes approach, $\bar{y}$ is the observed sample mean. Our prior specification is completed introducing the mapping function $h$. For the first continuous component we let $h_1$ be the identity function. For $h_2$ we define a thresholding function as in \citet{cana:duns:2011} which is defined in terms of thresholds partitioning the latent space. The partition of $\R$ can be chosen so to center the prior expectation on some particular probability mass function, but we let (suppressing the index $^{(j)}$ for simplicity) $A_k = [a_k, a_{k+1})$ with $a_0 = -\infty$ and $a_k = k$ for $k > 0$.

\subsection{Posterior computation}

We compute posterior quantities by means of Markov chain Monte Carlo (MCMC) sampling from the posterior distribution. Conditionally on the latent $y^*_i$ there is a rich variety of algorithms for posterior computation \citep{MacE:Mull:1998, neal:2000} for model \eqref{eq:model}. To take advantage of these approaches, we implement a Gibbs sampling algorithm which makes use of a data augmentation step which generates the latent $y_i^*$. Conditionally on such latent variables, we use Algorithm 8 in \citet{neal:2000} and, at each step of the sampler, compute the posterior quantities of interest. This approach follows the idea proposed in \citet{cana:duns:2011} and it is suitable for any discrete variables induced via thresholding functions $h_2$. In particular, for our crime data and a particular state, it consists in the following steps:
\begin{itemize}
		\item Generate 
		$u_{i} \sim U \Big(\Phi( a_{y_{2,i}} ; \tilde{\theta}_{i}, \tilde{\sigma}^2_{i} ), 
     			         \Phi( a_{y_{2,i}+1} ;   \tilde{\theta}_{i}, \tilde{\sigma}^2_{i}  )\Big)$ for $i = 1, \dots, n_j$, where
	\begin{align*}
		& \tilde{\theta}_{i}       = \theta_{S_i,2} + \Sigma_{S_i,21} \Sigma^{-1}_{S_i,11} (y_{1,i} - \theta_{S_i,1}) \\
		& \tilde{\sigma}^2_{i}  = \Sigma_{S_i,22} - \Sigma_{S_i,21} \Sigma^{-1}_{S_i,11} \Sigma_{S_i,12}
	\end{align*}
		are the usual conditional expectation and conditional variance of the multivariate normal.
		\item Let $y^*_{2,i} = \Phi^{-1}(u_{i};  \tilde{\theta}_{i}, \tilde{\sigma}^2_{i} )$ and  $y^*_{1,i} =  y_{1,i}$.
\end{itemize}
For each state, we run our sampler for 4,000 iterations and discard the first 1,000 as burn in. The traceplots of the marginal and joint distributions, computed for some points of the domain, suggest convergence and adequate mixing.

\subsection{Results}

In Table~\ref{tab:posteriors} we report some posterior summaries, namely the posterior mean of the quartiles of the marginal distributions of $y_1$ and of the conditional distributions of $y_1 | y_2 = 0$ and the marginal mean posterior $\text{pr}(y_2=0)$ and $\text{pr}(y_2 > 15)$. Most of the communities report zero murders. Such zero-inflation is automatically accommodated by our method through kernels located at negative values. This zero-inflation is a typical feature of many count data.
 
\begin{table}[ht]
\footnotesize \centering 
\caption{Posterior summaries for the marginal distributions and the conditionals distribution of $y_1 | y_2 = 0$ for the crime dataset}
\begin{tabular}{lrrrrrrrrrrr}
  \hline
State &\multicolumn{4}{c}{Marginal $f_1$} & \multicolumn{3}{c}{Marginal $f_2$} & \multicolumn{4}{c}{Conditional $f_{1|y_2=0}$} \\
 & $q_{0.25}$ & $q_{0.5}$ &$q_{0.75}$ & $E(y_1)$ &pr$(y_2=0)$ &pr$(y_2>15)$ &$E(y_2)$ &$q_{0.25}$ & $q_{0.5}$ &$q_{0.75}$ &$E(y_1|y_2=0)$ \\
AL &10.44 &11.86 &13.35 &13.07 &0.24 &0.09 &5.64 &10.01 &11.64 &13.46 &12.78 \\
AR &9.99 &11.16 &12.37 &11.25 &0.23 &0.11 &6.36 &9.72 &11.13 &12.56 &11.27 \\
AZ &10.84 &13.26 &14.85 &13.54 &0.18 &0.10 &5.77 &11.89 &13.90 &16.29 &14.33 \\
CA &12.21 &15.67 &20.21 &17.13 &0.12 &0.11 &6.82 &11.13 &14.86 &19.91 &16.12 \\
CO &11.99 &13.92 &15.62 &13.82 &0.32 &0.12 &6.44 &11.34 &13.37 &15.50 &13.58 \\
CT &17.20 &19.71 &23.40 &20.96 &0.46 &0.06 &3.33 &16.86 &19.25 &23.31 &20.35 \\
FL &12.57 &14.72 &17.54 &15.78 &0.16 &0.05 &4.30 &11.80 &14.05 &16.72 &14.64 \\
GA &10.42 &11.69 &14.85 &12.88 &0.19 &0.14 &6.26 &10.85 &12.43 &15.40 &13.26 \\
IA &11.85 &12.64 &13.87 &13.51 &0.42 &0.01 &2.32 &11.89 &12.75 &14.05 &13.52 \\
IL &15.01 &19.43 &23.49 &19.72 &0.35 &0.11 &3.78 &16.05 &19.88 &23.96 &20.14 \\
IN &11.19 &12.83 &15.03 &13.39 &0.40 &0.08 &5.33 &10.95 &12.47 &14.70 &12.86 \\
KY &10.59 &11.51 &12.59 &11.81 &0.39 &0.04 &3.27 &10.52 &11.54 &12.69 &11.80 \\
LA &8.62 &10.19 &11.71 &10.21 &0.19 &0.26 &12.38 &8.71 &10.32 &12.21 &10.36 \\
MA &15.23 &17.44 &20.85 &18.69 &0.57 &0.03 &1.70 &14.95 &17.12 &20.34 &18.18 \\
MI &11.85 &14.42 &17.05 &15.18 &0.37 &0.05 &3.04 &11.10 &13.70 &16.50 &14.31 \\
MN &13.03 &15.17 &17.62 &15.59 &0.65 &0.02 &1.77 &130 &15.20 &17.68 &15.58 \\
MO &11.28 &13.65 &16.70 &14.98 &0.33 &0.00 &1.82 &10.96 &13.28 &16.44 &14.59 \\
MS &9.70 &10.80 &12.99 &11.45 &0.06 &0.11 &8.81 &9.01 &10.50 &12.79 &10.88 \\
NC &11.14 &12.14 &13.54 &12.57 &0.12 &0.15 &6.75 &10.63 &11.78 &13.40 &12.33 \\
NH &14.33 &15.91 &17.81 &16.22 &0.64 &0.00 &1.04 &14.29 &15.92 &17.83 &16.22 \\
NJ &15.70 &190 &23.65 &20.05 &0.50 &0.04 &2.71 &14.86 &18.14 &22.69 &18.92 \\
NY &11.24 &12.91 &15.09 &14.47 &0.55 &0.11 &5.85 &11.49 &13.27 &15.58 &14.42 \\
OH &11.51 &13.43 &16.25 &14.56 &0.48 &0.05 &3.53 &11.21 &13.07 &15.73 &13.80 \\
OK &10.55 &11.73 &13.08 &11.94 &0.35 &0.07 &5.47 &10.18 &11.38 &12.71 &11.53 \\
OR &11.17 &12.37 &13.83 &13.15 &0.29 &0.03 &3.58 &10.94 &12.23 &13.82 &13.02 \\
PA &12.53 &15.64 &18.97 &15.94 &0.64 &0.02 &1.88 &11.51 &14.71 &180 &14.72 \\
RI &14.04 &15.57 &16.90 &15.88 &0.63 &0.04 &1.74 &13.74 &15.33 &16.67 &15.43 \\
SC &10.98 &12.74 &14.44 &12.81 &0.20 &0.07 &4.44 &10.97 &12.93 &14.76 &12.92 \\
TN &11.20 &12.50 &14.41 &13.46 &0.18 &0.08 &3.95 &10.83 &12.40 &14.60 &13.24 \\
TX &9.87 &11.86 &14.58 &12.62 &0.17 &0.06 &4.41 &9.94 &12.16 &15.18 &12.83 \\
UT &9.21 &10.32 &12.14 &10.74 &0.43 &0.04 &2.03 &9.15 &10.28 &12.13 &10.79 \\
VA &11.83 &13.26 &15.60 &14.17 &0.25 &0.21 &8.56 &12.05 &13.91 &17.09 &14.62 \\
WA &11.52 &13.40 &16.11 &14.32 &0.26 &0.07 &4.80 &11.88 &13.99 &16.80 &14.91 \\
WI &12.22 &14.11 &16.31 &14.68 &0.46 &0.00 &1.32 &12.48 &14.43 &16.65 &14.90 \\
 \hline
\end{tabular}
\label{tab:posteriors}
\end{table}

For sake of discussion, consider four states of the east coast, namely Connecticut, New Jersey, New York and Pennsylvania, whose posterior mean joint densities are plotted in Figure~\ref{fig:fourjoint}. The estimated joint densities are very different across states. For example, Connecticut presents a posterior mean density which is strongly multimodal for $y_1$, and particularly if we consider the conditional distribution of $y_1$ given $y_2=0$. Indeed, the nonparametric mixtures allow us to estimate conditional densities with different shapes for each of the infinite levels of the count variable. This is also clear from the estimated density for New York which is bimodal for $y_2=0$ and symmetric and unimodal for $y_2 >0$. New Jersey and Pennsylvania have unimodal conditional densities of $y_1$ for each level of $y_2$ with New Jersey also showing a mild skew-to-the right marginal density of $y_2$. Different modes in the marginal densities of $y_1$ may indicate different sub-populations with different economical status across the state.

\begin{figure*}[h]
       \centering
	\subfigure[]{\includegraphics[scale=.35]{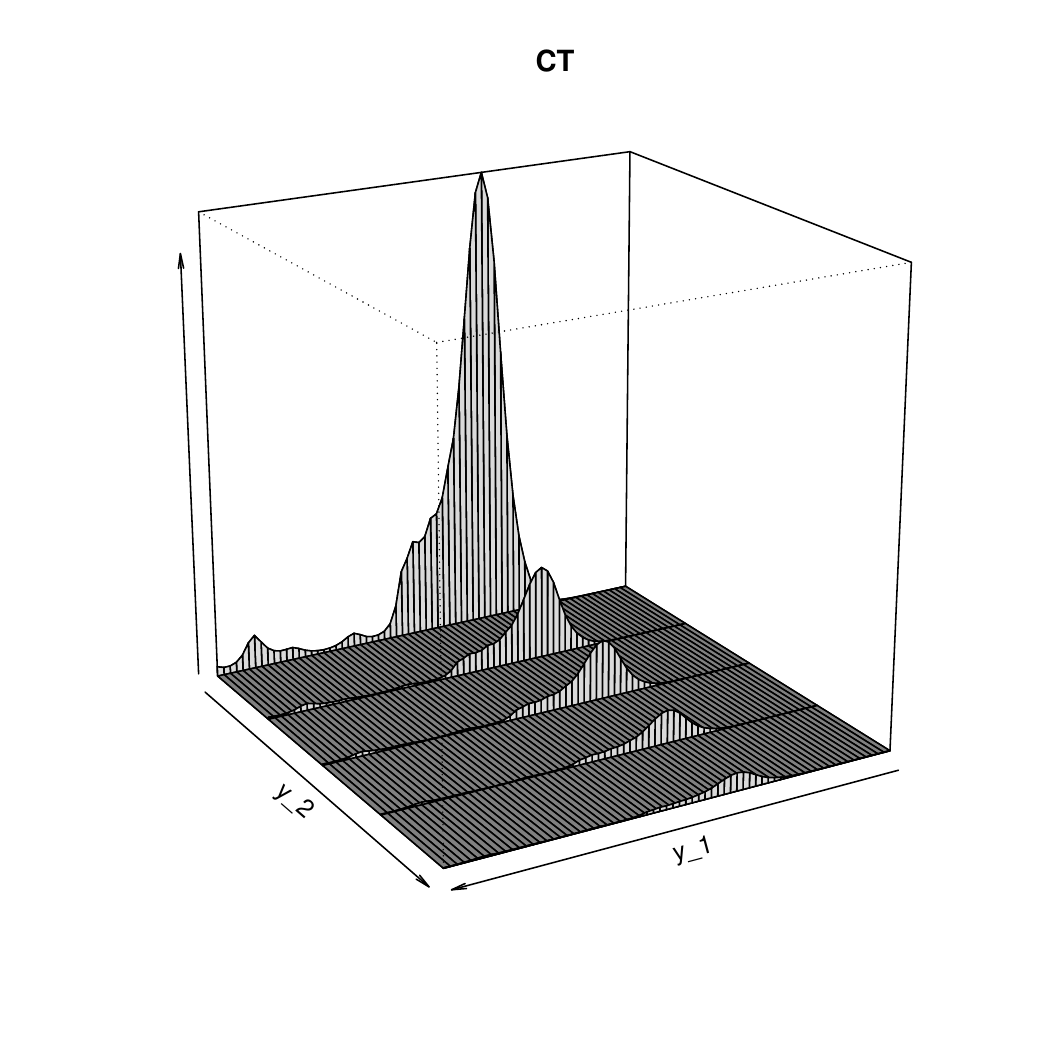}}
	\subfigure[]{\includegraphics[scale=.35]{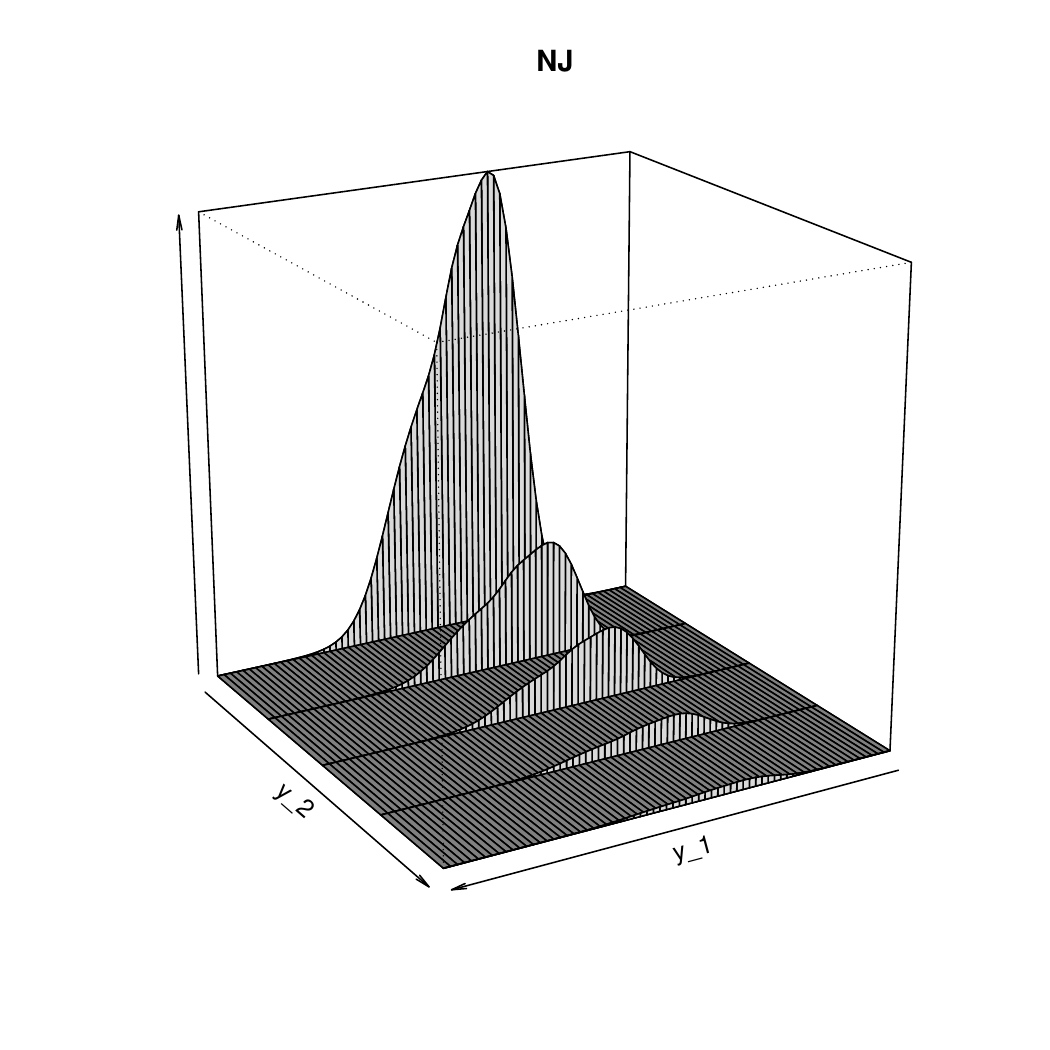}}
	\subfigure[]{\includegraphics[scale=.35]{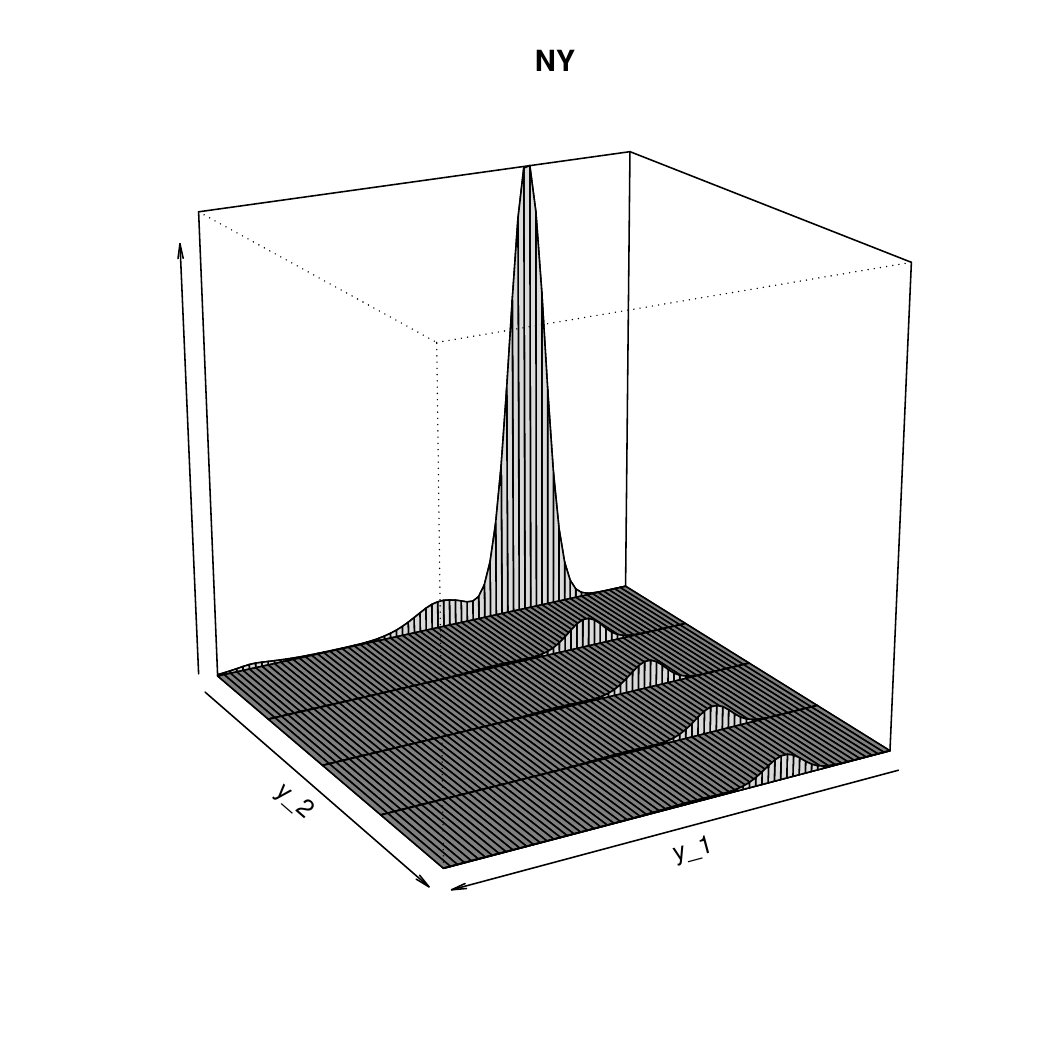}}
	\subfigure[]{\includegraphics[scale=.35]{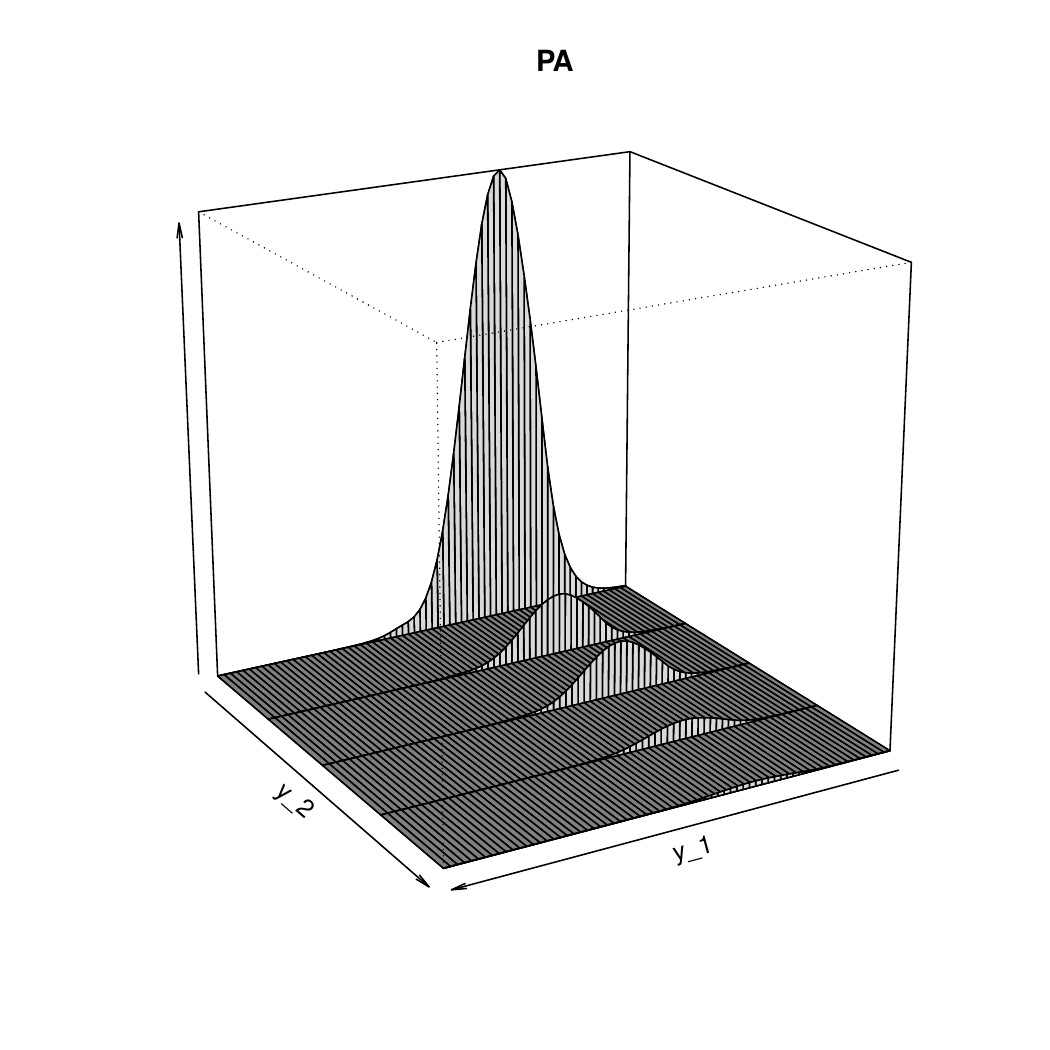}}
\caption{Mean posterior mixed-scale densities for Connecticut (CT), New Jersey (NJ), New York (NY), and Pennsylvania (PA) for $0<y_1<50$ and $y_2=0, \dots, 4$.}
\label{fig:fourjoint}
\end{figure*}

\appendix
\section{Proofs}

\begin{proof}[Proof of Theorem \ref{theo:consistency}]
The next two lemmas are useful to determine the size of the parameter space of $\mathcal F$, measured in terms of $L_1$ metric entropy. The first shows that the $L_1$ topology is maintained under the mapping $g$ and the second bounds the $L_1$ metric entropy of a sieve. 
\begin{lemma}
Assume that the true data generating density is $f_0 \in \mathcal{F}$. Choose any $f_0^*$ such that $f_0 = g(f^*_0)$. 
Let $U(f^*_0)= \{ f^* : ||f^*_0 - f^*|| < \epsilon\}$ be a $L_1$ neighborhood of size $\epsilon$ around $f^*_0$. 
Then the image $g(U(f_0^*))$ contains values $f \in \mathcal{F}$ in a $L_1$ neighborhood of $f_0$ of at most size $\epsilon$.
\label{lem:l1mapping}
\end{lemma}
The proof is omitted since it follows directly from the definition of $L_1$ neighborhood and from Fubini's theorem.
\begin{lemma}
Let $\mathcal{F}_n^* \subset \mathcal{F}^*$ denote a compact subset of $\mathcal{F}^*$, with $J(\delta,\mathcal{F}_n^*)$
the $L_1$ metric entropy corresponding to the logarithm of the minimum number of $\delta$-sized $L_1$ balls needed
to cover $\mathcal{F}_n^*$.  Letting $\mathcal{F}_n = g( \mathcal{F}_n^* )$, we have $J(\delta,\mathcal{F}_n) \le
J(\delta,\mathcal{F}_n^*)$.
\label{lem:metricentropy}
\end{lemma}
\begin{proof}[Proof of Lemma \ref{lem:metricentropy}]
Let $k = \exp\{ J(\delta,\mathcal{F}_n^*) \}$ be the number of $\delta$ balls needed to cover $\mathcal{F}_n^*$, with $f_1^*,\ldots,f_k^*$ denoting the centers of these balls so that $\mathcal{F}_n^* \subset \bigcup_{i=1}^k \mathcal{F}_{n,i}^*$, where $\mathcal{F}_{n,i}^* = \{ f^*: ||f^* - f_i^*|| < \delta \}$.  
From Lemma~\ref{lem:l1mapping}, it is clear we can define $\mathcal{F}_n \subset \bigcup_{i=1}^k \mathcal{F}_{n,i}$ where $\mathcal{F}_{n,i} = g( \mathcal{F}_{n,i}^* )$ is an $L_1$ neighborhood
around $f_i = g(f_i^*)$ of size at most $\delta$.  This defines a covering of $\mathcal{F}_n$ using $k$ $\delta$-sized $L_1$
balls, but this is not necessarily the minimal covering possible and hence $J(\delta,\mathcal{F}_n^*)$ provides an upper
bound on $J(\delta,\mathcal{F}_n)$.
\end{proof}

The rest of the proof follows along almost the same lines of \citet{ghos:etal:1999} in showing that the sets $\mathcal{F}_n \cap \{f: ||f-f_0||< \epsilon\}$ and $\mathcal{F}_n^{C}$ satisfy the conditions of an unpublished result of Barron (see Theorem 4.4.3 of \citet{book:ghos:rama}).
\end{proof}
\begin{proof}[Proof of Theorem \ref{theo:rate}]
Let $\mathcal{F}_n = g(\mathcal{F}^*_n)$. From Lemma~\ref{lem:metricentropy} we have $J(\delta,\mathcal{F}_n) \le J(\delta,\mathcal{F}_n^*)$.
Let $D(\epsilon, \mathcal{F})$ the $\epsilon$-packing number of $\mathcal F$, i.e. is the maximal number of points in $\mathcal F$ such that the distance between every pair is at least $\epsilon$. For every $\epsilon > \epsilon_n$, using $(iii)$ we have 
\[
\log D(\epsilon/2, \mathcal{F}) < \log D(\epsilon_n, \mathcal{F}^*) < C n \epsilon^2_n.
\]
Therefore applying Theorem 7.1 of \citet{ghos:etal:2000} with $j=1$, $D(\epsilon) = \exp(n \epsilon^2_n)$ and $\epsilon = M \epsilon_n$ with $M >2$ there exist a sequence of tests $\{\Phi_n\}$ that, for a universal constant $K$, satisfies
\begin{align}
& E_{f_0}\{ \Phi_n \} \leq \frac{\exp\{-(KM^2-1) n \epsilon^2_n\}}{1-\exp\{-KnM^2 \epsilon^2_n\}},\nonumber \\
& \sup_{f \in U^C \cap \mathcal{F}_n} E_{f}\{1 - \Phi_n \} \leq  \exp\{-K n M^2 \epsilon^2_n\}.
\label{eq:test}
\end{align}
The posterior probability assigned to $U^C$ can be written as 	
\begin{align*}
	 & \Pi \left\{ U^C \mid y_1, \dots, y_n	\right\}   = \\ 
		& \phantom{\Pi\{U^C\}} \frac{
			\int_{U^C \cap \mathcal{F}_n} \prod_{i=1}^n \frac{f(y_i)}{f_0(y_i)} \de \Pi(f)+
			\int_{U^C \cap \mathcal{F}_n^C} \prod_{i=1}^n \frac{f(y_i)}{f_0(y_i)} \de \Pi(f) 
                  }{
			\int \prod_{i=1}^n \frac{f(y_i)}{f_0(y_i)} \de \Pi(f) 
			}  \\
		& \phantom{\Pi\{U^C\}}
		\leq \Phi_n +  
		\frac{
			(1-\Phi_n)\int_{U^C \cap \mathcal{F}_n} \prod_{i=1}^n \frac{f(y_i)}{f_0(y_i)} \de \Pi(f)
			 }{
			\int \prod_{i=1}^n \frac{f(y_i)}{f_0(y_i)} \de \Pi(f) 
			}+ \\
		& \phantom{\Pi\{U^C\} \leq } + \frac{\int_{U^C \cap \mathcal{F}_n^C} \prod_{i=1}^n \frac{f(y_i)}{f_0(y_i)} \de \Pi(f)
                  }{
			\int \prod_{i=1}^n \frac{f(y_i)}{f_0(y_i)} \de \Pi(f) 
			} .
\end{align*}

Taking $KM^2 -1> K$ the first summand $E_{f_0}\{ \Phi_n \} \leq 2 \exp\{-K n \epsilon^2_n\}$ by (\ref{eq:test}).
The rest of the proof consists in proving that the remaining equation goes to zero in $P_{f_0}$-probability.
By Fubini's theorem and \eqref{eq:test} we have
\begin{align*}
E_{f_0}\left\{ (1-\Phi_n)\int_{U^C \cap \mathcal{F}_n} \prod_{i=1}^n \frac{f(y_i)}{f_0(y_i)} \de \Pi(f) \right\}  
& \leq  \sup_{f \in U^C \cap \mathcal{F}_n} E_{f}\{1 - \Phi_n \} \\
& \leq  \exp\{-K n M^2 \epsilon^2_n\},
\end{align*}
while by $(iv)$ we have 
\begin{align*}
E_{f_0}\left\{ \int_{U^C \cap \mathcal{F}_n^C} \prod_{i=1}^n \frac{f(y_i)} {f_0(y_i)} \de \Pi(f) \right\} 
& \leq \Pi(\mathcal{F}_n^C) \\
& = \Pi^*(\mathcal{F}_n^{*C})  \leq \exp\{- n \epsilon^2_n (C + 4) \}.
\end{align*}
The numerator of the second summand is hence exponentially small for $M > \sqrt{(C+4)/K}$.
Finally we need to lower bound the denomirator. Clearly $g(B^*_n) \subseteq B_n$ with
\[
 B_n = \left\{f : \int f_0 \log(f_0/f) d\mu \leq \epsilon^2_n, \int f_0 (\log(f_0/f))^2 d\mu \leq \epsilon^2_n \right\}
\]
and then $\Pi(B_n) \geq \Pi(g(B^*_n)) = \Pi^*(B^*_n)$ and using condition $(v)$ on $\Pi^*(B_n^*)$ we have
\begin{eqnarray*}
	& \int_{B_n} \int f_0 \log(f_0/f) d\mu d\Pi(f)  \leq \int_{B_n} \epsilon_n^2 d\Pi(f) \\
	& \int_{B_n} \int f_0 \left(\log(f_0/f)\right)^2 d\mu d\Pi(f)  \leq \int_{B_n} \epsilon_n^2 d\Pi(f),
\end{eqnarray*}
and hence
\begin{align*}
 \int \prod_{i=1}^n \frac{f(y_i)}{f_0(y_i)} \de \Pi(f)  
 &\geq \int_{B_n} \prod_{i=1}^n \frac{f(y_i)}{f_0(y_i)} \de \Pi(f) \\
 &\geq \exp(-2n \epsilon^2_n) \Pi(B_n) \\
 &\geq \exp(-2n \epsilon^2_n) \Pi^*(B_n^*) \\
 & \geq \exp\{-n\epsilon^2_n (C+2)\}
\end{align*}
Then using Lemma 8.1 of \citet{ghos:etal:2000} we obtain 
\[
	E_{P_0} \int \prod_{i=1}^n \frac{f(y_i)}{f_0(y_i)} \de \Pi(f) \to 1
\]
that concludes the proof.
\end{proof}

\begin{proof}[Proof of Lemma~\ref{lem:minimax}]
If $q_j < \infty$ for all $j = 1, \dots, p_2$, the $p_2$ categorical variables can be combined into a single categorical variable, say $\tilde{y}_2$, with $q= \prod_{j=1}^{p_2} q_j$ levels. To estimate the probability mass function of a categorical variables with finite number of levels, the minimax rate is $n^{-1/2}$, i.e, for $n\to \infty$
\begin{equation*}
	  |p_{0j} - \hat{p}_{j} | =  O(n^{-1/2}),
\end{equation*}
where $\hat{p}_{j}$  and $p_{0j}$ are the point estimate and true marginal probability masses for level $j$, respectively.
Since also $q$ is finite, the density of the $p_1$ continuous variables can be estimated conditionally on each level of $\tilde{y}_2$. 
The minimax optimal rate for each conditional density is clearly $n^{-\alpha/(2\alpha + p_1)}$, i.e., for $n\to \infty$
\begin{equation*}
	 \int_{\mathcal{Y}} |f_0(y_1| \tilde{y}_2 = j) - \hat{f}(y_1|\tilde{y}_2 = j)| d y_1 =  O(n^{-\frac{\alpha}{2\alpha+p_1}}),
\end{equation*}
where $\hat{f}(y_1|\tilde{y}_2 = j)$ is a point estimate of the conditional density for $y_1$ given $\tilde{y_2} = j$ and $f_0(y_1|\tilde{y}_2 = j)$ is the true conditional density. For fixed $\tilde{y_2} = j$, we have
\begin{align*}
	 \int_{\mathcal{Y}} |f_0(y_1, & \tilde{y}_2 )  - \hat{f}(y_1, \tilde{y}_2)| d y_1  \\
			= &\int_{\mathcal{Y}} \left|f_0(y_1, \tilde{y}_2 )  - \hat{f}(y_1, \tilde{y}_2) \pm f_0(y_1, \tilde{y}_2) \frac{\hat{p}_j}{p_{0j}} \right| d y_1  \\
			\leq& \int_{\mathcal{Y}} \left|f_0(y_1, \tilde{y}_2 )  - f_0(y_1, \tilde{y}_2) \frac{\hat{p}_j}{p_{0j}} \right | d y_1 \\ 
	& + \int_{\mathcal{Y}} \left| f_0(y_1, \tilde{y}_2) \frac{\hat{p}_j}{p_{0j}} - \hat{f}(y_1, \tilde{y}_2) \right| d y_1  \\
	= & \int_{\mathcal{Y}} \left| \frac{\hat{p}_j - p_{0j}}{p_{0j}} f_0(y_1, \tilde{y}_2 ) \right| d y_1 \\ 
	& + \int_{\mathcal{Y}} \left| \hat{f}(y_1| \tilde{y}_2) - f_0(y_1| \tilde{y}_2)  \right| d y_1  \\
	= &  \,\, O(n^{-1/2}) + O(n^{\frac{\alpha}{2\alpha+p_1}}) =  O(n^{-\frac{\alpha}{2\alpha+p_1}}).
\end{align*}
Hence the minimax optimal rate for the joint density is $n^{-\alpha/(2\alpha + p_1)}$.
\end{proof}

\section*{Acknowledgements}

The authors would like to thank the reviewers for their comments that help improve the manuscript. This research was partially supported by grant R01 ES017240-01 from the National Institute of Environmental Health Sciences of the National Institutes of Health.

\bibliographystyle{apalike}
\bibliography{biblio}

\end{document}